\documentclass[a4paper]{article}

\usepackage{amsmath,amssymb,color}

\usepackage{graphicx}
 \usepackage{subfigure}
\DeclareGraphicsExtensions{.pdf,.eps}

\newcommand{\bu}{\boldsymbol u}

\newcommand{\bv}{\boldsymbol v}
\newcommand{\bV}{\boldsymbol V}
\newcommand{\bw}{\boldsymbol w}
\newcommand{\bbeta}{\boldsymbol  \eta}

\newcommand{\be}{\boldsymbol e}

\newcommand{\bvar}{\boldsymbol \varphi}

\newcommand{\bU}{\boldsymbol U}

\newcommand{\bff}{\boldsymbol f}

\newtheorem{Theorem}{Theorem}[section]

\newtheorem{remark}[Theorem]{Remark}

\newtheorem{Proof}{{\em Proof:}}

\newenvironment{proof}{\begin{Proof}\rm}{\hfill $\Box$ \end{Proof}}
\usepackage{hyperref}

\usepackage{macros}

\title{On the influence of the nonlinear term in the numerical approximation of Incompressible Flows by means
of proper orthogonal decomposition methods}

\author{Bosco
Garc\'{\i}a-Archilla\thanks{Departamento de Matem\'atica Aplicada
II, Universidad de Sevilla, Sevilla, Spain. Research is supported by
Spanish MCINYU under grants PGC2018-096265-B-I00 and PID2019-104141GB-I00 (bosco@esi.us.es)}
  \and Julia Novo\thanks{Departamento de
Matem\'aticas, Universidad Aut\'onoma de Madrid, Spain. Research is supported
by Spanish MINECO
under grants PID2019-104141GB-I00 and VA169P20  (julia.novo@uam.es)}\and
Samuele Rubino\thanks{Department EDAN \& IMUS, Universidad de Sevilla, Spain. Research is supported by
Spanish MCINYU under grant RTI2018-093521-B-C31  (samuele@us.es)}}
\date{\today}
\begin{document}
\maketitle

\begin{abstract}
 We consider proper orthogonal decomposition (POD) methods to approximate the incompressible Navier-Stokes equations. We
 study the case in which one discretization for the nonlinear term is used in the snapshots (that are computed with a full order
 method (FOM)) and a different discretization of the nonlinear term is applied in the POD method. We prove that
 an additional error term appears in this case, compared with the case in which the same discretization of the
 nonlinear term is applied for both the FOM and the POD methods. However, the added term has the same size as
 the error coming from the FOM so that the rate of convergence of the POD method is barely affected. We analyze
  the case in which we add grad-div stabilization to both the FOM and the POD methods because it allows
  to get error bounds with constants independent of inverse powers of the viscosity. We also study the
  case in which no stabilization is added. Some numerical experiments support the theoretical analysis.
 \end{abstract}

\noindent{\bf AMS subject classifications.} 35Q30,  65M12, 65M15, 65M20, 65M60, 65M70,\\ 76B75. \\
\noindent{\bf Keywords.} Navier-Stokes equations, proper orthogonal decomposition, nonlinear term discretization, grad-div stabilization.

\section{Introduction}
The computational cost of direct numerical simulations can be reduced by using reduced order models. The proper orthogonal decomposition (POD) method is based on a reduced basis (snapshots) that are computed by means of a full order method (FOM). 
In this paper, we study the numerical approximation of incompressible flows with POD methods.

We consider the Navier-Stokes equations
\begin{align}
\label{NS} \partial_t\bu -\nu \Delta \bu + (\bu\cdot\nabla)\bu + \nabla p &= \bff &&\text{in }\ (0,T]\times\Omega,\nonumber\\
\nabla \cdot \bu &=0&&\text{in }\ (0,T]\times\Omega,
\end{align}
in a bounded domain $\Omega \subset {\mathbb R}^d$, $d \in \{2,3\}$ with initial condition $\bu(0)=\bu^0$. In~\eqref{NS},
$\bu$ is the velocity field, $p$ the kinematic pressure, $\nu>0$ the kinematic viscosity coefficient,
 and $\bff$ represents the accelerations due to external body forces acting
on the fluid. The Navier-Stokes equations \eqref{NS} must be complemented with boundary conditions. For simplicity,
we only consider homogeneous
Dirichlet boundary conditions $\bu = \boldsymbol 0$ on $\partial \Omega$.

In practical simulations one can apply some given software to compute the snapshots. It could then be the case that a different discretization
is used for the discretization of the nonlinear term in the FOM method and the POD method (the last one being typically implemented by means of a hand-made code instead of an existing one). Our aim in this paper is to study the influence of the use
of different discretizations for the nonlinear terms on the final error bounds of the POD method. 

In \cite{novo_rubino} POD stabilized methods for the Navier-Stokes equations were considered and analyzed for a case in which
the snapshots are based on a non inf-sup stable method and a case in which the snapshots are based on an inf-sup stable method. In the present paper, we consider the second case. Our snapshots are based on an inf-sup stable method and, as in \cite{novo_rubino}, we add
grad-div stabilization to both the FOM and the POD methods. We  analyze the case in which different
discretizations are applied for the nonlinear term. Adding grad-div stabilization we are able to prove error bounds for the method
with constants that do not depend explicitly on inverse powers of the viscosity, although, as usual, may depend on it through the norms
of the theoretical solution.

In the recent preprint \cite{archiv_incon} we have found an study of the consistency of the nonlinear discretization in FOM and 
POD methods. In this reference no stabilization is included neither in the FOM nor in the POD method. The authors of \cite{archiv_incon}
conclude that the use of different discretizations in the nonlinear term yields additional terms that prevent the POD method
from recovering the FOM accuracy. In the present paper, we prove, in agreement with the results in \cite{archiv_incon}, that
there are some additional terms in the error bound of the final POD method coming from the use of different discretizations in the nonlinear terms. We prove that the additional terms have the size of the $L^2$ error in the velocity and the $L^2$ error of the divergence of the velocity of the FOM method. This holds both for the methods with grad-div stabilization and also for the plain methods considered in 
\cite{archiv_incon}. We then conclude that there is no inconsistency in the error bounds using different discretization for the nonlinear terms in the FOM and POD methods. 

Following \cite{archiv_incon}, we analyzed a concrete case in which the so called divergence or skew-symmetric form of the nonlinear term (commonly used in practice) is
used for the FOM method and the EMAC form, see \cite{emac}, is applied for the POD method. The error analysis for any other combination
of discretizations
of the nonlinear terms with analogous properties could be
carried out in a similar way. The EMAC formulation was designed to conserve energy, momentum and angular momentum. Considering
a continuous in time method it is easy to see, as we prove in Section 4, that the effect of grad-div stabilization in a method with EMAC form for the nonlinear term is a lost of kinetic energy. The method maintains however the conservation of the
 momentum and angular momentum. In practice, as stated in \cite{emac_fully}, a fully version of a method with EMAC form for the nonlinear term 
is applied. In \cite{emac_fully} it is proved that using Newton method for the nonlinear term is the option from which more quantities
are still conserved, although the conservation of the kinetic energy is lost. In this sense, adding grad-div stabilization incurres
in the same phenomena so that it seems not to be a hard problem since either kinetic energy or any other properties will be
lost with the EMAC formulation in the fully discrete case. 

The outline of the paper is as follows. In Section 2 we introduce some notation. In Section 3 we state some
preliminaries concerning the POD method. Section 4 is devoted to the error analysis in which one considers
different discretizations for the nonlinear term in the FOM and POD methods. In Section 5 we present some numerical experiments.
We end the paper with some conclusions.
\section{Preliminaries and notation}\label{sec:PN}
The following Sobolev embeddings \cite{Adams} will be used in the analysis: For
$q \in [1, \infty)$, there exists a constant $C=C(\Omega, q)$ such
that
\begin{equation}\label{sob1}
\|v\|_{L^{q'}} \le C \| v\|_{W^{s,q}}, \,\,\quad
\frac{1}{q'}
\ge \frac{1}{q}-\frac{s}{d}>0,\quad q<\infty, \quad v \in
W^{s,q}(\Omega)^{d}.
\end{equation}
The following inequality can be found in  \cite[Remark 3.35]{John}
\begin{eqnarray}\label{diver_vol}
\|\nabla \cdot \bv\|_0\le \|\nabla   \bv \|_0,\quad \bv\in H_0^1(\Omega)^d,
\end{eqnarray}
{where, here and in the sequel, we use the notation $\|\cdot\|_j$ for $\|\cdot\|_{W^{j,2}}=\|\cdot\|_{H^j}$.}
Let us denote by $Q=L_0^2(\Omega)=\left\{q\in L^2(\Omega)\mid (q,1)=0\right\}$.
Let $\mathcal{T}_{h}=(\tau_j^h,\phi_{j}^{h})_{j \in J_{h}}$, $h>0$ be a family of partitions of $\overline\Omega$, where $h$ denotes the maximum diameter of the elements $\tau_j^h\in \mathcal{T}_{h}$, and $\phi_j^h$ are the mappings from the reference simplex $\tau_0$ onto $\tau_j^h$.
We shall assume that the partitions are shape-regular and quasi-uniform. 
{We define the following finite element spaces
\begin{eqnarray*}
Y_h^l&=& \left\{v_h\in C^0(\overline\Omega)\mid {v_h}_{\mid_K}\in {\Bbb P}_l(K),\quad \forall K\in \mathcal T_h\right\}, \ l\ge 1,\nonumber\\
{\boldsymbol Y}_h^l&=&(Y_h^l)^d,\quad {\boldsymbol X}_h^l={\boldsymbol Y}_h^l\cap H_0^1(\Omega)^d, \nonumber\\
Q_h^l&=&Y_h^l\cap L_0^2(\Omega).
\end{eqnarray*}
\begin{eqnarray}\label{eq:V}
{\boldsymbol V}_{h,l}={\boldsymbol X}_h^l\cap \left\{ {\boldsymbol \chi}_{h} \in H_0^1(\Omega)^d \mid
(q_{h}, \nabla\cdot{\boldsymbol\chi}_{h}) =0  \quad\forall q_{h} \in Q_{h}^{l-1}
\right\},\quad l\geq 2.
\end{eqnarray}}If the family of
meshes is quasi-uniform then  the following inverse
inequality holds for each {$\bv_{h} \in Y_{h}^{l}$}, see e.g., \cite[Theorem 3.2.6]{Cia78},
\begin{equation}
\label{inv} \| \bv_{h} \|_{W^{m,p}(K)} \leq c_{\mathrm{inv}}
h_K^{n-m-d\left(\frac{1}{q}-\frac{1}{p}\right)}
\|\bv_{h}\|_{W^{n,q}(K)},
\end{equation}
where $0\leq n \leq m \leq 1$, $1\leq q \leq p \leq \infty$, and $h_K$
is the diameter of~$K \in \mathcal T_h$.
Let $l \geq 2$, we consider the  MFE pair known as Hood--Taylor elements \cite{BF,hood0} $({\boldsymbol X}_h^l, Q_{h}^{l-1})$.

For these elements a uniform inf-sup condition is satisfied (see \cite{BF}), that is, there exists a constant $\beta_{\rm is}>0$ independent of the mesh size $h$ such that
\begin{equation}\label{lbbh}
 \inf_{q_{h}\in Q_{h}^{l-1}}\sup_{\bv_{h}\in{\boldsymbol X}_h^l}
\frac{(q_{h},\nabla \cdot \bv_{h})}{\|\bv_{h}\|_{1}
\|q_{h}\|_{L^2/{\mathbb R}}} \geq \beta_{\rm{is}}.
\end{equation}
We will denote by $P_{Q_h}$ the $L^2$ orthogonal projection onto $Q_h^{l-1}$. The following bound holds
\begin{eqnarray}\label{eq:propre}
\|q-P_{Q_h}q\|_0\le C h^l \|q\|_l,\quad \forall q\in H^l(\Omega).
\end{eqnarray}
As a direct method, as in \cite[Section 5]{novo_rubino}, we consider a Galerkin method with grad-div stabilization and 
for simplicity in the analysis we consider the implicit Euler method in time. Let us fix $T>0$ and $M>0$ and take $\Delta t=T/M$.
The method reads as follows: given $\bu_h^0\approx \bu_0$ 
find $(\bu_h^n,p_h^n)\in {\boldsymbol X}_h^l\times Q_h^{l-1}$ for $n\ge 1$ such that 
\begin{eqnarray}\label{eq:gal_grad_div}
&&\left(\frac{\bu_h^{n}-\bu_h^{n-1}}{\Delta t },\bv_h\right)+\nu(\nabla \bu_h^n,\nabla \bv_h)+b_h(\bu_h^n,\bu_h^n,\bv_h)-(p_h^n,\nabla \cdot \bv_h)
\nonumber\\
&&\quad+
\mu(\nabla \cdot\bu_h^n,\nabla \cdot \bv_h)=({\boldsymbol f^n},\bv_h) \quad \forall \bv_h\in {\boldsymbol X}_h^l,\\
&&(\nabla \cdot \bu_h^n,q_h)=0 \quad \forall q_h\in Q_h^{l-1},\nonumber
\end{eqnarray}
where $\mu$ is the positive grad-div stabilization parameter.

For the discretization of the nonlinear term we consider the following form, the so called divergence or
skew-symmetric form,
\begin{equation}\label{nonli_div}
b_h(\bu,\bv,\bw)=((\bu\cdot \nabla \bv),\bw)+\frac{1}{2}((\nabla \cdot \bu)\bv,\bw).
\end{equation}

It is well-known that considering the discrete divergence-free space $\bV_{h,l}$ we can remove the pressure from \eqref{eq:gal_grad_div} since
$\bu_h^n\in \bV_{h,l}$ satisfies for $n\ge 1$
\begin{eqnarray}\label{eq:gal_grad_div2}
&&\left(\frac{\bu_h^{n}-\bu_h^{n-1}}{\Delta t },\bv_h\right)+\nu(\nabla \bu_h^n,\nabla \bv_h)+b_h(\bu_h^n,\bu_h^n,\bv_h)
\nonumber\\
&&\quad+
\mu(\nabla \cdot\bu_h^n,\nabla \cdot \bv_h)=({\boldsymbol f^n},\bv_h), \quad \forall \bv_h\in {\bV}_{h,l}.
\end{eqnarray}
For this method the following bound holds, see \cite{NS_grad_div}
\begin{eqnarray}\label{eq:cota_grad_div}
\|\bu^n-\bu_h^n\|_0+ \left(\mu\sum_{j=1}^M \Delta t \|\nabla \cdot (\bu^n-\bu_h^n)\|_0^2\right)^{1/2}\le C(\bu,p,l+1) \left(h^{l}+\Delta t \right),
\end{eqnarray}
for $1\le n\le M$, where the constant $C(\bu,p,l+1)$ depends on $\|\bu\|_{L^\infty(H^{l+1})}$, $\left(\int_0^T\|\bu_t\|^2_l\right)^{1/2}$,
$\left(\int_0^T\|\bu_{tt}\|^2_0\right)^{1/2}$ and $\|p\|_{L^{\infty}(H^l)}$ but does not depend explicitly on inverse powers of $\nu$.

For the plain Galerkin method the following bound holds where the error constants depend explicitly on inverse powers of $\nu$
\begin{eqnarray}\label{eq:cota_sin_grad_div}
\|\bu^n-\bu_h^n\|_0+h\|\nabla (\bu^n-\bu_h^n)\|_0\le C(\bu,p,\nu^{-1},l+1) \left(h^{l+1}+\Delta t \right).
\end{eqnarray} 
\section{Proper orthogonal decomposition}\label{sec:POD}
We will consider a proper orthogonal decomposition (POD) method.
As for the FOM we fix $T>0$ and $M>0$ and take $\Delta t=T/M$. We consider the following space
$$
{\cal \bU}=<\bu_h^1,\ldots,\bu_h^M>,
$$
where $\bu_h^j=\bu_h(\cdot,t_j)$.
Let $d_v$ be the dimension of the space $\cal \bU$.

Let $K_v$ be the correlation matrix corresponding to the snapshots $K_v=((k_{i,j}^v))\in {\mathbb R}^{M\times M}$,
 where
$$
k_{i,j}^v=\frac{1}{M}(\bu_h^i,\bu_h^j),
$$
and $(\cdot,\cdot)$ is the inner product in $L^2(\Omega)^d$. Following \cite{kunisch} we denote by
$ \lambda_1\ge  \lambda_2,\ldots\ge \lambda_{d_v}>0$ the positive eigenvalues of $K_v$ and by
$\bv_1,\ldots,\bv_{d_v}\in {\mathbb R}^{M}$ the associated eigenvectors.  Then, the (orthonormal) POD bases are given by
\begin{eqnarray}\label{lachi}
\bvar_k=\frac{1}{\sqrt{M}}\frac{1}{\sqrt{\lambda_k}}\sum_{j=1}^{M} v_k^j \bu_h(\cdot,t_j),\end{eqnarray}
where $v_k^j$ is the $j$-th component of the eigenvector $\bv_k$ and the following error formulas hold, see \cite[Proposition~1]{kunisch}
\begin{eqnarray}\label{eq:cota_pod_0}
\frac{1}{M}\sum_{j=0}^{M}\left\|\bu_h^j-\sum_{k=1}^r(\bu_h^j,\bvar_k)\bvar_k\right\|_{0}^2&\le&\sum_{k=r+1}^{d_v}\lambda_k.
\end{eqnarray}
We will denote by $S^v$ the stiffness matrix for the POD basis: $S^v=((s_{i,j}^v))\in {\mathbb R}^{d_v\times d_v}$, $s_{i,j}^v=(\nabla \bvar_i,\nabla \bvar_j)$. In that case, for any $\bv \in {\cal \bU}$,
the following inverse inequality holds, see \cite[Lemma 2]{kunisch}
\begin{equation}\label{eq:inv_S}
||\nabla \bv ||_0\le \sqrt{\|S^v\|_2}\|\bv\|_0.
\end{equation}

In the sequel we will denote by
$$
{\cal \bU}^r=<\bvar_1,\bvar_2,\ldots,\bvar_r>,
$$
and by $P_r^v$,  the $L^2$-orthogonal projection onto ${\cal \bU}^r$.

\subsection{A priori bounds for the orthogonal projection onto ${\cal \bU}^r$}
In \cite[Section 3.1]{novo_rubino} some a priori bounds for the FOM and for the orthogonal projection were obtained
assuming for the solution of \eqref{NS} the same regularity needed to prove \eqref{eq:cota_grad_div}.
For $0\le j\le M$ the following bound holds (see \cite[(31)]{novo_rubino})
\begin{eqnarray}\label{eq:uh_infty}
\|\bu_h^j\|_\infty\le
 C_{\bu,{\rm inf}}.
\end{eqnarray}
For $1\le j\le M$ the following bounds hold (see \cite[(39), (40), (41)]{novo_rubino})
\begin{eqnarray}\label{eq:cotaPlinf2}
\|P_r^v \bu_h^j\|_\infty&\le& C_{\rm inf},\\
\label{eq:cotanablaPlinf}
\|\nabla P_r^v \bu_h^j\|_\infty&\le& C_{1,\rm inf},
\label{eq:cotanablaPld}\\
\|\nabla P_r^v\bu^j\|_{L^{2d/(d-1)}}&\le& C_{\rm ld}.
\end{eqnarray}
\section{The POD method}\label{sec:NA}
We now consider the  grad-div POD model. As for the full order model, for simplicity in the error analysis, we also use
the implicit Euler method as time integrator. Taking an initial approximation $\bu_r^0\in  {\cal \bU}^r$, for $n\ge 1$, find $\bu_r^n\in {\cal \bU}^r$ such that
\begin{eqnarray}\label{eq:pod_method2}
&&\left(\frac{\bu_r^{n}-\bu_r^{n-1}}{\Delta t},\bvar\right)+\nu(\nabla \bu_r^n,\nabla\bvar)+b_{pod}(\bu_r^n,\bu_r^n,\bvar)
+\mu(\nabla \cdot\bu_r^n,\nabla \cdot\bvar)\nonumber\\
&&\quad=(\bff^{n},\bvar),\quad \forall \bvar\in {\cal \bU}^r,
\end{eqnarray}
where we observe that  ${\cal \bU}^r\subset \bV_{h,l}$ so that there is no pressure approximation in \eqref{eq:pod_method2}.
For the discretization of the nonlinear term we consider the EMAC form
\begin{equation}\label{nonli_emac}
b_{pod}(\bu,\bv,\bw)=(2D(\bu)\bv,\bw)+((\nabla \cdot \bu)\bv,\bw),
\end{equation}
where the deformation tensor is
$$
D(\bu)=\frac{1}{2}\left(\nabla \bu+(\nabla \bu)^T\right),
$$
and the second term in \eqref{nonli_emac} is added to ensure the property
\begin{equation}\label{cancel}
b_{pod}(\bv,\bv,\bv)=0.
\end{equation}
The EMAC form is based on the identity
\begin{equation}\label{eq:ide}
(\bu\cdot \nabla)\bu=2D(\bu)\bu-\frac{1}{2}\nabla |\bu|^2.
\end{equation}
From \eqref{eq:ide} it can be observed that the EMAC form of the nonlinear term implies a modification of the pressure, i.e., in
a velocity-pressure formulation of a method with the EMAC form for the nonlinear term the pressure approximation converges to
$p-\frac{1}{2}|\bu|^2$ instead of to the original pressure $p$ in \eqref{NS}.

Considering for simplicity the continuous in time case, i.e., the method:
\begin{eqnarray*}
\left(\bu_{r,t},\bvar\right)+\nu(\nabla \bu_r,\nabla\bvar)+b_{pod}(\bu_r,\bu_r,\bvar)
+\mu(\nabla \cdot\bu_r,\nabla \cdot\bvar)=(\bff,\bvar),
\end{eqnarray*}
taking $\bvar=\bu_r$ and assuming as in \cite{emac}, \cite{emac_fully}, $\bff=0$ and $\nu=0$ one gets
applying \eqref{cancel}
\begin{eqnarray}\label{lost}
\frac{d}{dt}\|\bu_r\|_0^2=-\mu\|\nabla \cdot\bu_r\|_0^2.
\end{eqnarray}
This means that the grad-div stabilization produces a lost in the kinetic energy. As stated in the introduction, considering
a fully discrete method with Newton discretization for the nonlinear term one has also the effect of a lost of kinetic energy, see
\cite{emac_fully}.
On the other hand, it is easy to prove that the grad-div term has no negative effect in the conservation of linear and angular momentum
so that both are still conserved. To prove this, one can argue as in \cite[Theorem 3.3]{Ingimarson+2021} and observe that the test functions
used to achieve the conservation of linear and angular momentum have divergence zero so that the grad-div term does not affect the
proof.
\subsection{Error analysis of the method}
Denoting by
$$
\bbeta_h^n=P_r^v\bu_h^n-\bu_h^n,
$$
it is easy to get
\begin{eqnarray}\label{eq:prov_prop2}
&&\left(\frac{P_r^v \bu^{n}_h-P_r^v \bu^{n-1}_h}{\Delta t},\bvar\right)+\nu(\nabla  P_r^v \bu^n_h,\nabla\bvar)+b_h(P_r^v \bu^n_h,P_r^v \bu^n_h,\bvar)
\nonumber\\
&&\quad+\mu(\nabla \cdot P_r^v \bu^n_h,\nabla \cdot\bvar)\nonumber\\
&&=(\bff^{n},\bvar)+\nu(\nabla \bbeta_h^n,\nabla \bvar)+\mu(\nabla \cdot\bbeta_h^n,\nabla \cdot \bvar)\\
&&\quad+b_h(P_r^v\bu^n_h,P_r^v \bu^n_h,\bvar)-b_h(\bu^{n}_h,\bu^{n}_h,\bvar),\quad \forall \bvar\in {\cal \bU}^r.\nonumber
\end{eqnarray}
Subtracting \eqref{eq:prov_prop2} from \eqref{eq:pod_method2} and denoting   by
$$\be_r^n=\bu_r^n-P_r^v\bu_h^n$$
we get $\forall \bvar\in {\cal U}^r$
\begin{eqnarray}\label{eq:er1}
&&\left(\frac{\be_r^n-\be_r^{n-1}}{\Delta t },\bvar\right)+\nu(\nabla \be_r^n,\nabla \bvar)+\mu(\nabla \cdot \be_r^n,\nabla \cdot \bvar)\\
&&=\left(b_h(\bu_h^n,\bu_h^n,\bvar)-b_{pod}(\bu_r^n,\bu_r^n,\bvar)\right)-\nu(\nabla \bbeta_h^n,\nabla\bvar)-\mu(\nabla \cdot \bbeta_h^n,\nabla \cdot\bvar).\nonumber
\end{eqnarray}
Taking $\bvar=\be_r^n$ we obtain
\begin{eqnarray}\label{eq:error2}
&&\frac{1}{2\Delta t}\left(\|\be_r^n\|_0^2-\|\be_r^{n-1}\|_0^2\right)+\nu\|\nabla \be_r^n\|_0^2
+\mu\|\nabla \cdot \be_r^n\|_0^2\\
&&\quad\le\left(b_h(\bu_h^n,\bu_h^n,\be_r^n)-b_{pod}(\bu_r^n,\bu_r^n,\be_r^n)\right)-\nu(\nabla \bbeta_h^n,\nabla\be_r^n)-\mu(\nabla \cdot \bbeta_h^n,\nabla \cdot\be_r^n)\nonumber\\
&&\quad =I+II+III.\nonumber
\end{eqnarray}
We will bound the first term on the right-hand side above.
Arguing as in \cite{archiv_incon} we get
\begin{eqnarray}\label{eq:incon1}
|I|&\le& \left|b_{pod}(\bu_h^n,\bu_h^n,\be_r^n)-b_{pod}(\bu_r^n,\bu_r^n,\be_r^n)\right|\nonumber\\
&&\quad+\left|b_h(\bu_h^n,\bu_h^n,\be_r^n)-b_{pod}(\bu_h^n,\bu_h^n,\be_r^n)\right|.
\end{eqnarray}
To bound the first term on the right-hand side of \eqref{eq:incon1} we argue as in \cite{ols_reb} (see also \cite{review}).
We first observe that
\begin{eqnarray}\label{52}
&&b_{pod}(\bu_h^n,\bu_h^n,\be_r^n)-b_{pod}(\bu_r^n,\bu_r^n,\be_r^n)=
b_{pod}(-\bbeta_h^n,\bu_h^n,\be_r^n)\nonumber\\
&&\quad+b_{pod}(P_r^v\bu_h^n,-\bbeta_h^n,\be_r^n)-b_{pod}(P_r^v\bu_h^n,\be_r^n,\be_r^n)
\nonumber\\
&&\quad-b_{pod}(\be_r^n,P_r^v\bu_h^n,\be_r^n)-b_{pod}(\be_r^n,\be_r^n,\be_r^n).
\end{eqnarray}
The last term on the right-hand side of \eqref{52} vanishes due to the property \eqref{cancel}. The expression of the fourth one
is
$$
b_{pod}(\be_r^n,P_r^v\bu_h^n,\be_r^n) = ((P_r^v\bu_h^n\cdot\nabla)\be_r^n,\be_r^n)+((\be_r^n\cdot \nabla )\be_r^n,P_r^v \bu_h^n) + ((\nabla \cdot \be_r^n)\be_r^n,P_r^v\bu_h^n).
$$
Integration by parts reveals that the first term on the right-hand side above equals $-\frac{1}{2}((\nabla \cdot P_r^v\bu_h^n)\be_r^n,\be_r^n)$, whilst the other two equal $-((\be_r^n\cdot \nabla )P_r^v\bu_h^n,\be_r^n)$, so that
$$
b_{pod}(\be_r^n,P_r^v\bu_h^n,\be_r^n) =-\frac{1}{2}((\nabla \cdot P_r^v\bu_h^n)\be_r^n,\be_r^n) 
-((\be_r^n\cdot \nabla )P_r^v\bu_h^n,\be_r^n).
$$
Since the third term on the right-hand side of~\eqref{52} is
$$
b_{pod}(P_r^v\bu_h^n,\be_r^n,\be_r^n)= 2((\be_r^n\cdot \nabla )P_r^v\bu_h^n,\be_r^n) + ((\nabla \cdot P_r^v\bu_h^n)\be_r^n,\be_r^n),
$$
we see that the last three terms on the right-hand side of~\eqref{52} add to 
$-((\be_r^n\cdot \nabla )P_r^v\bu_h^n,\be_r^n)-\frac{1}{2}((\nabla \cdot P_r^v\bu_h^n)\be_r^n,\be_r^n)$, so that~\eqref{52} can be written as
%Applying the definition of the deformation tensor and integrating by parts we get
%\begin{eqnarray*}
%&&2(D(\be_r^n)P_r^v\bu_h^n,\be_r^n)=((\nabla \be_r^n)P_r^v\bu_h^n,\be_r^n)+(P_r^v\bu_h^n,(\nabla \be_r^n)\be_r^n)
%\\
%&&\quad=((P_r^v\bu_h^n\cdot\nabla)\be_r^n,\be_r^n)+((\be_r^n\cdot \nabla )\be_r^n,P_r^v \bu_h^n)\\
%&&\quad=-((\nabla \cdot P_r^v\bu_h^n)\be_r^n,\be_r^n)-((P_r^v\bu_h^n\cdot \nabla )\be_r^n,\be_r^n)
%-((\nabla \cdot \be_r^n)\be_r^n,P_r^v\bu_h^n)\nonumber\\
%&&\quad\quad-((\be_r^n\cdot \nabla )P_r^v\bu_h^n,\be_r^n)\nonumber\\
%&&\quad=-\frac{1}{2}((\nabla \cdot P_r^v\bu_h^n)\be_r^n,\be_r^n)-((\nabla \cdot \be_r^n)\be_r^n,P_r^v\bu_h^n)-((\be_r^n\cdot \nabla )P_r^v\bu_h^n,\be_r^n)\\
%&&\quad=-\frac{1}{2}((\nabla \cdot P_r^v\bu_h^n)\be_r^n,\be_r^n)-((\nabla \cdot \be_r^n)\be_r^n,P_r^v\bu_h^n)
%-(D(P_r^v\bu_h^n)\be_r^n,\be_r^n).
%\end{eqnarray*}
%Inserting the above expression in \eqref{52} we obtain
\begin{eqnarray}\label{53}
&&b_{pod}(\bu_h^n,\bu_h^n,\be_r^n)-b_{pod}(\bu_r^n,\bu_r^n,\be_r^n)=
b_{pod}(-\bbeta_h^n,\bu_h^n,\be_r^n)\\
&&\quad+b_{pod}(P_r^v\bu_h^n,-\bbeta_h^n,\be_r^n)-\frac{1}{2}((\nabla \cdot P_r^v\bu_h^n)\be_r^n,\be_r^n)
-(D(P_r^v\bu_h^n)\be_r^n,\be_r^n).\nonumber
\end{eqnarray}
For the first term on the right-hand side of \eqref{53}, applying  \eqref{eq:uh_infty},
we obtain
\begin{eqnarray}\label{53_uno}
b_{pod}(-\bbeta_h^n,\bu_h^n,\be_r^n)\le C\|\bbeta_h^n\|_{1}\|\bu_h^n\|_{\infty}\|\be_r^n\|_0
\le CC_{\bu,{\rm inf}}\|\bbeta_h^n\|_{1}\|\be_r^n\|_0.
\end{eqnarray}
Applying H$\ddot{\rm o}$lder's inequality, \eqref{eq:cotanablaPld} and Sobolev embedding \eqref{sob1} we get for the second term
on the right-hand side of \eqref{53}
\begin{eqnarray}\label{53_dos}
b_{pod}(P_r^v\bu_h^n,-\bbeta_h^n,\be_r^n)
&\le& 2\|D(P_r^v\bu_h^n)\|_{L^{2d/(d-1)}}\|\bbeta_h^n\|_{L^{2d}}\|\be_r^n\|_0\nonumber\\
&&\quad+\|\nabla \cdot P_r^v\bu_h^n\|_{L^{2d/(d-1)}}\|\bbeta_h^n\|_{L^{2d}}\|\be_r^n\|_0\nonumber\\
&\le& C \|\nabla P_r^v\bu_h^n\|_{L^{2d/(d-1)}}\|\bbeta_h^n\|_{L^{2d}}\|\be_r^n\|_0\nonumber\\
&\le& C C_{\rm ld}\|\bbeta_h^n\|_{1}\|\be_r^n\|_0.
\end{eqnarray}
Finally, for the last two terms on the right-hand side of \eqref{53}, applying \eqref{eq:cotanablaPlinf}, we get
\begin{eqnarray}\label{53_tres}
-\frac{1}{2}((\nabla \cdot P_r^v\bu_h^n)\be_r^n,\be_r^n)
-(D(P_r^v\bu_h^n)\be_r^n,\be_r^n)\le C C_{1,\rm inf}\|\be_r^n\|_0^2.
\end{eqnarray}
From \eqref{53}, \eqref{53_uno}, \eqref{53_dos} and \eqref{53_tres} and Poincar\'e inequality we finally reach
\begin{eqnarray}\label{eq:incon11}
&&|b_{pod}(\bu_h^n,\bu_h^n,\be_r^n)-b_{pod}(\bu_r^n,\bu_r^n,\be_r^n)|
\nonumber\\
&&\quad\le (1+CC_{1,\rm inf})\|\be_r^n\|_0^2+C( C^2_{\bu,{\rm inf}}+C^2_{\rm ld})\|\nabla \bbeta_h^n\|_0^2.
\end{eqnarray}

For the second term on the right-hand side of \eqref{eq:incon1},
we recall that
\begin{align*}
b_h(\bu_h^n,\bu_h^n,\be_r^n) & = (\bu_h^n\cdot\nabla \bu_h^n,\be_r^n)+ \frac{1}{2}((\nabla \cdot \bu_h^n)\bu_h^n,\be_r^n).
\\
b_{pod}(\bu_h^n,\bu_h^n,\be_r^n) & = (\bu_h^n\cdot\nabla \bu_h^n,\be_r^n)+ (\be_r^n\cdot\nabla \bu_h^n,\bu_h^n)+
((\nabla \cdot \bu_h^n)\bu_h^n,\be_r^n).
\end{align*}
We notice that both
$b_h(\bu_h^n,\bu_h^n,\be_r^n)$ and $b_{pod}(\bu_h^n,\bu_h^n,\be_r^n)$ share the term
$(\bu_h^n\cdot\nabla \bu_h^n,\be_r^n)$, and both also have the term $((\nabla \cdot \bu_h^n)\bu_h^n,\be_r^n)$ but with a factor~$1/2$ in the second one,  so that we have
$$
b_h(\bu_h^n,\bu_h^n,\be_r^n)-b_{pod}(\bu_h^n,\bu_h^n,\be_r^n) = -(\be_r^n\cdot\nabla \bu_h^n,\bu_h^n)  
-\frac{1}{2}((\nabla \cdot \bu_h^n)\bu_h^n,\be_r^n)
$$
Now, integrating by parts the first term on the right-hand side above we have
$$
b_h(\bu_h^n,\bu_h^n,\be_r^n)-b_{pod}(\bu_h^n,\bu_h^n,\be_r^n) = \frac{1}{2}(|\bu_h^n|^2,\nabla \cdot \be_r^n) 
-\frac{1}{2}((\nabla \cdot \bu_h^n)\bu_h^n,\be_r^n),
$$
and, since $\nabla\cdot\bu=0$, we can write
\begin{equation}
\label{bosco01}
b_h(\bu_h^n,\bu_h^n,\be_r^n)-b_{pod}(\bu_h^n,\bu_h^n,\be_r^n) = \frac{1}{2}(|\bu_h^n|^2,\nabla \cdot \be_r^n) 
+\frac{1}{2}((\nabla \cdot(\bu^n -\bu_h^n))\bu_h^n,\be_r^n),
\end{equation}
For the second term on the right-hand side of~\eqref{bosco01} we have
\begin{eqnarray}\label{eq:2t3}
\frac{1}{2}((\nabla \cdot(\bu^n -\bu_h^n))\bu_h^n,\be_r^n) &\le&\frac{1}{2}\|\bu_h^n\|_\infty \|\nabla \cdot (\bu_h^n-\bu^n)\|_0\|\be_r^n\|_0\\
&\le&C_{\bu,{\rm inf}}^2\frac{\mu}{4}\|\nabla \cdot (\bu_h^n-\bu^n)\|_0^2+\frac{1}{4\mu}\|\be_r^n\|_0^2,\nonumber
\end{eqnarray}
For the first term on the right-hand side of~\eqref{bosco01} we obtain
\begin{eqnarray}\label{36}
&&\frac{1}{2}(|\bu_h^n|^2,\nabla \cdot \be_r^n)|\le \frac{1}{2}(|\bu_h^n|^2-|\bu^n|^2,\nabla \cdot \be_r^n)|+ \frac{1}{2} |(|\bu^n|^2,\nabla \cdot \be_r^n)|\nonumber\\
&&= \frac{1}{2}|(|\bu_h^n|^2-|\bu^n|^2,\nabla \cdot \be_r^n)|+ \frac{1}{2} |((I-P_{Q_h})|\bu^n|^2,\nabla \cdot \be_r^n)
|\label{ahora_si}\\
&&\le \frac{1}{2} \|\bu^n+\bu_h^n\|_{\infty}\|\bu^n-\bu_h^n\|_0\|\nabla \cdot \be_r^n\|_0+\frac{1}{2}\|(I-P_{Q_h})|\bu^n|^2\|_0\|\nabla \cdot \be_r^n\|_0.\nonumber
%\\
%&&\quad +\||\bu^n|^2-|\bu_h^n|^2\|_0\|\nabla \cdot \be_r^n\|_0,
\end{eqnarray}
%where in the last inequality we have applied the $L^2$ stability of the $P_{Q_h}$ projection. 
Then, applying \eqref{eq:uh_infty}
\begin{eqnarray}\label{eq:2t4}
&&\frac{1}{2}|(|\bu_h^n|^2,\nabla \cdot \be_r^n)|\le \frac{1}{2}(\|\bu^n\|_\infty+C_{\bu,{\rm inf}})\|\bu^n-\bu_h^n\|_0\|\nabla \cdot \be_r^n\|_0
\nonumber\\
&&\quad+
\frac{1}{2}\|(I-P_{Q_h})|\bu^n|^2\|_0\|\nabla \cdot \be_r^n\|_0\le C\mu^{-1}(\|\bu^n-\bu_h^n\|_0^2+\|(I-P_{Q_h})|\bu^n|^2\|_0^2)\nonumber\\
&&\quad +\frac{\mu}{4}\|\nabla \cdot \be_r^n\|_0^2,
\end{eqnarray}
where the generic constat $C$ above depends on $\|\bu\|_{L^\infty(L^\infty)}$ and $C_{\bu,{\rm inf}}$.
Inserting \eqref{eq:incon11}, \eqref{bosco01}, \eqref{eq:2t3} and \eqref{eq:2t4} into \eqref{eq:incon1} we get
\begin{eqnarray*}
|I|&\le& \left(1+CC_{1,\rm inf}+\frac{1}{4\mu}\right)\|\be_r^n\|_0^2+\frac{\mu}{4}\|\nabla \cdot \be_r^n\|_0^2
\nonumber\\
&&\quad+C( C^2_{\bu,{\rm inf}}+C^2_{\rm ld})\|\nabla \bbeta_h^n\|_0^2
+C_{\bu,{\rm inf}}^2\frac{\mu}{4}\|\nabla \cdot (\bu_h^n-\bu^n)\|_0^2\nonumber\\
&&\quad
+ C\mu^{-1}(\|\bu-\bu_h\|_0^2+\|(I-P_{Q_h})|\bu^n|^2\|_0^2).
\end{eqnarray*}
Including into a generic constant $C$ the dependence on $C_{\bu,{\rm inf}}$ and $C_{\rm ld}$ of the third and fourth terms above 
we may write
\begin{eqnarray}\label{cotaI}
|I|&\le& \left(1+CC_{1,\rm inf}+\frac{1}{4\mu}\right)\|\be_r^n\|_0^2+\frac{\mu}{4}\|\nabla \cdot \be_r^n\|_0^2
\nonumber\\
&&\quad+C\|\nabla \bbeta_h^n\|_0^2
+C{\mu}\|\nabla \cdot (\bu_h^n-\bu^n)\|_0^2\nonumber\\
&&\quad
+ C\mu^{-1}(\|\bu-\bu_h\|_0^2+\|(I-P_{Q_h})|\bu^n|^2\|_0^2).
\end{eqnarray}
We also have
\begin{eqnarray}\label{cotaII}
|II|\le \frac{\nu}{2}\|\nabla \bbeta_h^n\|_0^2+\frac{\nu}{2}\|\nabla \be_r^n\|_0^2,
\end{eqnarray}
and
\begin{eqnarray}\label{cotaIII}
|III|\le \|\nabla \cdot\bbeta_h^n\|_0^2+\frac{\mu}{4}\|\be_r^n\|_0^2.
\end{eqnarray}
Inserting \eqref{cotaI}, \eqref{cotaII} and \eqref{cotaIII} into \eqref{eq:error2} and adding terms we get
\begin{eqnarray}\label{eq:fin1}
&&\|\be_r^n\|_0^2+\nu\sum_{j=1}^n\Delta t \|\nabla \be_r^j\|_0^2
+\mu\sum_{j=1}^n\Delta t\|\nabla \cdot \be_r^j\|_0^2\nonumber\\
&&\le \|\be_r^0\|_0^2+\sum_{j=1}^n\Delta t \left(1+CC_{1,\rm inf}+\frac{1}{4\mu}\right)\|\be_r^j\|_0^2
+C(\nu+\mu)\sum_{j=1}^n\Delta t\|\nabla \bbeta_h^j\|_0^2\nonumber\\
&&\quad +
C\mu\sum_{j=1}^n\Delta t\|\nabla\cdot (\bu_h^j-\bu^j)\|_0^2
+
C\mu^{-1}\sum_{j=1}^n\Delta t\|\bu_h^j-\bu^j\|_0^2\nonumber\\
&&\quad +C\mu^{-1}\sum_{j=1}^n\Delta t \|(I-P_{Q_h})|\bu^j|^2\|_0^2.
\end{eqnarray}
Assuming
\begin{equation}\label{eq_Cu1}
\Delta t C_u:=\Delta t(1+CC_{1,\rm inf}+\frac{1}{4\mu})\le \frac{1}{2}
\end{equation}
and applying Gronwall's Lemma \cite[Lemma 5.1]{hey_ran_IV} and \eqref{eq:cota_pod_0}, 
\eqref{eq:inv_S}, \eqref{eq:cota_grad_div} and \eqref{eq:propre}
we obtain
\begin{eqnarray}\label{eq:error3_b}
&&\|\be_r^n\|_0^2+\nu\sum_{j=1}^n\Delta t \|\nabla \be_r^j\|_0^2+\mu\sum_{j=1}^n\Delta t\|\nabla \cdot \be_r^j\|_0^2\nonumber\\
&&\quad\le e^{2T C_u}\left(\|\be_r^0\|_0^2+CT(\nu+\mu)\|S^v\|_0\sum_{k=r+1}^{d_v}\lambda_k\right.\\
&&\quad \left.
+C(1+\mu^{-1}T)C(\bu,p,l+1)^2 \left(h^{2l}+(\Delta t)^2 \right)+C\mu^{-1}h^{2l}T\||\bu|^2\|_{L^\infty(H^l)}\right).\nonumber
\end{eqnarray}
\begin{remark}
Let us observe that the error terms in the last line of \eqref{eq:error3_b} arise from using different discretization for the nonlinear term in the FOM and the POD methods. We also observe that these errors are of the same size as the error of the FOM method.
\end{remark}
\begin{Theorem}\label{th:prin}
Let $\bu$ be the velocity in the Navier-Stokes equations \eqref{NS}, let $\bu_r$ be the grad-div POD
stabilized approximation defined in \eqref{eq:pod_method2}, assume that the solution $(\bu,p)$ of \eqref{NS} is regular enough
and condition \eqref{eq_Cu1} holds. Then, the error can be bounded as follows
\begin{eqnarray}\label{eq:cota_finalSUPv}
\sum_{j=1}^n\Delta t \|\bu_r^j-\bu^j\|_0^2&\le& 3Te^{2T C_u}\left(\|\be_r^0\|_0^2+CT(\nu+\mu)\|S^v\|_0\sum_{k=r+1}^{d_v}\lambda_k\right.\nonumber\\
&&\quad \left.
+C(1+\mu^{-1}T)C(\bu,p,l+1)^2 \left(h^{2l}+(\Delta t)^2 \right)\right.\nonumber\\
&&\quad\left.+C\mu^{-1}h^{2l}T\||\bu|^2\|_{L^\infty(H^l)}\right)+3T\sum_{k=r+1}^{d_v}\lambda_k\nonumber\\
&&\quad +3T C(\bu,p,l+1)^2 \left(h^{2l}+(\Delta t)^2 \right).
\end{eqnarray}
\end{Theorem}
\begin{proof}
Since $\sum_{j=1}^n\Delta t \|\be_r^j\|_0^2\le T \max_{1\le j\le n}\|\be_r^j\|_0^2$ and
\begin{eqnarray*}\label{eq:lauso}
\sum_{j=1}^n\Delta t \|\bu_r^j-\bu^j\|_0^2&\le&
3\left(\sum_{j=1}^n\Delta t \|\be_r^j\|_0^2+\sum_{j=1}^n\Delta t\|P_r^v \bu_h^j-\bu_h^j)\|_0^2\right.\nonumber\\
&&\quad\left.+\sum_{j=1}^n\Delta t\| \bu_h^j-\bu^j\|_0^2\right),
\end{eqnarray*}
 from  \eqref{eq:error3_b}, \eqref{eq:cota_pod_0} and \eqref{eq:cota_grad_div}
we easily obtain \eqref{eq:cota_finalSUPv}.
\end{proof}
\begin{remark}
We have chosen, as in \cite{archiv_incon}, to analyze the concrete case in which the divergence form is used for the FOM method
and the EMAC form is applied for the POD method but the error analysis of any other combination of discretizations
of the nonlinear terms with analogous properties could be
carried out in a similar way.
\end{remark}
\begin{remark}
With the POD method defined in \eqref{eq:pod_method2} we can obtain an approximation to the velocity but not to the pressure. In case an approximation to the pressure is also required one can use a supremizer pressure recovery method. This procedure is analyzed in \cite{novo_rubino}, see also \cite{schneier}. One can argue exactly as in \cite[Theorem 5.4]{novo_rubino} to get an error bound for the pressure, in which, as in Theorem \ref{th:prin}, one will have to add to the error terms in \cite[Theorem 5.4]{novo_rubino} those coming from the different discretizations used in the nonlinear terms (i.e. fourth and fifth terms on the right-hand side of \eqref{cotaI}). 
\end{remark}
\subsection{Error analysis of the method without grad-div stabilization}
In this section we consider the case analyzed in \cite{archiv_incon} in which the plain Galerkin method is used instead of the grad-div stabilized method, i.e, we take $\mu=0$ both in equations \eqref{eq:gal_grad_div2} and \eqref{eq:pod_method2}.

 In the following error analysis, as in \cite{archiv_incon}, one does not get error bounds independent on inverse powers of the viscosity.
 The analysis can be obtained with slight modifications of the analysis of the previous section.

%Firstly, the last two terms one on  the left and the other on the right hand side of \eqref{eq:er1} do not appear and
First, we notice that the last terms on each side of identity~\eqref{eq:er1} are not present. Then, applying \eqref{diver_vol} in~\eqref{36} to bound $\|\nabla \cdot \be_r^n\|_0$ by~$\|\nabla \be_r^n\|_0$, instead of~\eqref{eq:2t4} we get
%On the one hand, as in the previous analysis, the final error bound will depend on both the $L^2$ error of the FOM velocity approximation and %the
%error in the divergence of the FOM approximation (that 
%has at least the rate of convergence of the velocity error in the $H^1$ norm). 
%applying \eqref{diver_vol} we can bound $\|\nabla \cdot \be_r^n\|_0\le\|\nabla \be_r^n\|_0$ in \eqref{36} to get instead of \eqref{eq:2t4}
\begin{eqnarray}\label{eq:2t4_b}
&&\frac{1}{2}|(|\bu_h^n|^2,\nabla \cdot \be_r^n)|\le \frac{1}{2}(\|\bu^n\|_\infty+C_{\bu,{\rm inf}})\|\bu^n-\bu_h^n\|_0\|\nabla \cdot \be_r^n\|_0
\nonumber\\
&&\quad+
\frac{1}{2}\|(I-P_{Q_h})|\bu^n|^2\|_0\|\nabla \cdot \be_r^n\|_0\le C\nu^{-1}(\|\bu^n-\bu_h^n\|_0^2+\|(I-P_{Q_h})|\bu^n|^2\|_0^2)\nonumber\\
&&\quad +\frac{\nu}{4}\|\nabla \cdot \be_r^n\|_0^2.
\end{eqnarray}
Consequently, instead of
\eqref{cotaI} we obtain
\begin{eqnarray*}
|I|&\le& \left(1+CC_{1,\rm inf}+\frac{1}{4}\right)\|\be_r^n\|_0^2+\frac{\nu}{4}\|\nabla \be_r^n\|_0^2
\nonumber\\
&&\quad+C\|\nabla \bbeta_h^n\|_0^2
+C\|\nabla \cdot (\bu_h^n-\bu^n)\|_0^2\nonumber\\
&&\quad
+ C\nu^{-1}(\|\bu-\bu_h\|_0^2+\|(I-P_{Q_h})|\bu^n|^2\|_0^2).
\end{eqnarray*}
Instead of \eqref{cotaII} we can write
\begin{eqnarray*}
|II|\le {\nu}\|\nabla \bbeta_h^n\|_0^2+\frac{\nu}{4}\|\nabla \be_r^n\|_0^2,
\end{eqnarray*}
%can be absorbed into the left hand side of error equation with the term $\nu\|\nabla \be_r^n\|_0$ with the price of
%a constant $\nu^{-1}$ multiplying the error bounds. 
to conclude the following inequality instead of~\eqref{eq:fin1}, 
\begin{eqnarray}
\label{otra_mas}
&&\|\be_r^n\|_0^2+\nu\sum_{j=1}^n\Delta t \|\nabla \be_r^j\|_0^2\nonumber\\
&&\le \|\be_r^0\|_0^2+\sum_{j=1}^n\Delta t \left(5/4+CC_{1,\rm inf}\right)\|\be_r^j\|_0^2
+C\nu\sum_{j=1}^n\Delta t\|\nabla \bbeta_h^j\|_0^2
%\nonumber
\\
&&\quad +
C\sum_{j=1}^n\Delta t\|\nabla\cdot (\bu_h^j-\bu^j)\|_0^2
+
C\nu^{-1}\sum_{j=1}^n\Delta t\|\bu_h^j-\bu^j\|_0^2\nonumber\\
&&\quad +C\nu^{-1}\sum_{j=1}^n\Delta t \|(I-P_{Q_h})|\bu^j|^2\|_0^2.\nonumber
\end{eqnarray}
From the above inequality one can argue exactly as before to conclude that, if 
\begin{equation}\label{eq:sin_esta}
\Delta t C_u:=\Delta t \left(5/4+CC_{1,\rm inf}\right)\le \frac{1}{2},
\end{equation}
then, the following bound holds
\begin{eqnarray}\label{eq:error3_bb}
&&\|\be_r^n\|_0^2+\nu\sum_{j=1}^n\Delta t \|\nabla \be_r^j\|_0^2\le e^{2T C_u}\left(\|\be_r^0\|_0^2+CT\nu\|S^v\|_0\sum_{k=r+1}^{d_v}\lambda_k\right.\nonumber\\
&&\quad \left.
+CT(h^{-2}+\nu^{-1})C(\bu,p,\nu^{-1},l+1)^2 \left(h^{2l+2}+(\Delta t)^2 \right)\right.\nonumber\\
&&\quad\left.+C\nu^{-1}h^{2l}T\||\bu|^2\|_{L^\infty(H^l)}\right),
\end{eqnarray}
where we have applied \eqref{diver_vol} to bound the fourth term on the right-hand side~of~\eqref{otra_mas} and the error bound of the plain Galerkin method \eqref{eq:cota_sin_grad_div}.
From \eqref{eq:error3_bb} and arguing as before we conclude
\begin{Theorem}\label{th:prinb}
Let $\bu$ be the velocity in the Navier-Stokes equations \eqref{NS}, let $\bu_r$ be the grad-div POD approximation without stabilization (case $\mu=0$) and assume that the solution $(\bu,p)$ of \eqref{NS} is regular enough. Then, assuming condition \eqref{eq:sin_esta} holds, the error
can be bounded as follows
\begin{eqnarray}\label{eq:cota_finalSUPvb}
\sum_{j=1}^n\Delta t \|\bu_r^j-\bu^j\|_0^2&\le& 3Te^{2T C_u}\left(\|\be_r^0\|_0^2+CT\nu\|S^v\|_0\sum_{k=r+1}^{d_v}\lambda_k\right.\nonumber\\
&&\quad \left.
+CT(h^{-2}+\nu^{-1})C(\bu,p,\nu^{-1}, l+1)^2 \left(h^{2l+2}+(\Delta t)^2 \right)\right.\nonumber\\
&&\quad\left.+C\nu^{-1}h^{2l}T\||\bu|^2\|_{L^\infty(H^l)}\right)+3T\sum_{k=r+1}^{d_v}\lambda_k\nonumber\\
&&\quad +3T C(\bu,p,\nu^{-1},l+1)^2 \left(h^{2l+2}+(\Delta t)^2 \right).
\end{eqnarray}
\end{Theorem}
The final error bound
depends explicitly on $\nu^{-1}$ and on both the $L^2$ error of the FOM velocity approximation and the
error in the divergence of the FOM approximation (that we have bounded by the $H^1$ error). 
\begin{remark}
Alternatively, we can also bound the method without grad-div stabilization in a way that inverse constants of the viscosity do not appear
explicitly in the constants. However, the error depends on the error of the plain Galerkin method \eqref{eq:cota_sin_grad_div} which in terms depend on $\nu^{-1}$.

To this end, instead of \eqref{36}, we can write, integrating by parts
\begin{eqnarray*}
&&\frac{1}{2}(|\bu_h^n|^2,\nabla \cdot \be_r^n)|\le \frac{1}{2}(|\bu_h^n|^2-|\bu^n|^2,\nabla \cdot \be_r^n)|+ \frac{1}{2} |(|\bu^n|^2,\nabla \cdot \be_r^n)|\nonumber\\
&&= \frac{1}{2}|(\nabla\left(|\bu_h^n|^2-|\bu^n|^2\right),\be_r^n)|+ \frac{1}{2} |\nabla\left((I-P_{Q_h})|\bu^n|^2\right),\be_r^n)
|\\
&&\le C \|\bu^n+\bu_h^n\|_{1,\infty}\|\bu^n-\bu_h^n\|_1\|\be_r^n\|_0+\frac{1}{2}\|(I-P_{Q_h})|\bu^n|^2\|_1\|\be_r^n\|_0\nonumber\\
&&\le C \left(\|\bu^n-\bu_h^n\|_1^2+\|(I-P_{Q_h})|\bu^n|^2\|_1^2\right)+\frac{1}{2}\|\be_r^n\|_0^2.
\end{eqnarray*}
Then, applying \eqref{diver_vol}, instead of \eqref{cotaI} we get
\begin{eqnarray*}
|I|&\le& \left(1+CC_{1,\rm inf}\right)\|\be_r^n\|_0^2+C\|\nabla \bbeta_h^n\|_0^2
+C\|\bu_h^n-\bu^n\|_1^2\nonumber\\
&&\quad
+ C\|(I-P_{Q_h})|\bu^n|^2\|_1^2),
\end{eqnarray*}
and then, for $C_u=1+CC_{1,\rm inf}$ and $\Delta t C_u\le 1/2$,  we conclude
\begin{eqnarray*}
&&\|\be_r^n\|_0^2+\nu\sum_{j=1}^n\Delta t \|\nabla \be_r^j\|_0^2\le e^{2T C_u}\left(\|\be_r^0\|_0^2+CT\nu\|S^v\|_0\sum_{k=r+1}^{d_v}\lambda_k\right.\\
&&\quad \left.
+CTC(\bu,p,\nu^{-1},l+1)^2 h^{-2}\left(h^{2l+2}+(\Delta t)^2 \right)\right.\\
&&\quad\left.+Ch^{2(l-1)}T\||\bu|^2\|_{L^\infty(H^l)}\right).\nonumber
\end{eqnarray*}
As pointed out in \cite{review}, from the above error bound we observe that the grad-div stabilization could be suppressed when we use the EMAC form of the nonlinear
term in the ROM method and we can still obtain error bounds with constants independent of inverse powers of $\nu$, apart
from the dependence thought the error of the plain Galerkin method. However, the price to be paid is a lower rate of convergence. Compare the
last term on the right-hand side above with \eqref{eq:cota_finalSUPv}. The rate of convergence obtained with
this argument is also lower than the bound \eqref{eq:cota_finalSUPvb}.
\end{remark}

%\section{Numerical experiments}

\section{Numerical experiments}\label{sec:num}

In this section, we present numerical results for the grad-div reduced order model (ROM) \eqref{eq:pod_method2}, introduced and analyzed in the previous section. Actually, a comparison is performed by using both the skew-symmetric \eqref{nonli_div} and the EMAC \eqref{nonli_emac} form for the discretization of the nonlinear term at ROM level. The numerical experiments are performed on the benchmark problem of the 2D unsteady flow around a cylinder with circular cross-section \cite{SchaferTurek96} at Reynolds number $Re=100$ ($\nu=10^{-3}\,\rm{m^{2}/s}$). The open-source FE software FreeFEM \cite{Hecht12} has been used to run the numerical experiments, following the setup from \cite{novo_rubino}.

\medskip 

{\em FOM and POD modes.}
The numerical method used to compute the snapshots for the grad-div-ROM \eqref{eq:pod_method2} is the grad-div finite element method (FEM) \eqref{eq:gal_grad_div} described in Section \ref{sec:PN}, with the skew-symmetric form \eqref{nonli_div} of the nonlinear term. A spatial discretization using the Hood--Taylor MFE pair ${\bf P}^{2}-\mathbb{P}^1$ for velocity-pressure is considered on a relatively coarse computational grid (see \cite{novo_rubino}), for which $h = 2.76\cdot 10^{-2}\,\rm{m}$, resulting in $32\,488$ d.o.f. for velocities and $4\,151$ d.o.f. for pressure. %For the grad-div-FEM \eqref{eq:gal_grad_div}, we have considered $\mu=10^{-2}\,\rm{m^{2}/s}$.

As in \cite{novo_rubino}, for the time discretization, a semi-implicit Backward Differentiation Formula of order two (BDF2) has been applied (see \cite{AhmedRubino19} for further details), with time step $\Delta t = 2\cdot 10^{-3}\,\rm{s}$. %In particular, we have considered an extrapolation for the convection velocity by means of Newton--Gregory backward polynomials \cite{Cellier91}. Without entering into the details of the derivation, for which we refer the reader to e.g. \cite{Cellier91}, we consider the following extrapolation of order two for the discrete velocity: $\widehat{\uv}_{h}^{n}=2\uhv^{n}-\uhv^{n-1}$, $n\geq 1$, in order to achieve a second-order accuracy in time. For the initialization $(n=0)$, we have considered $\uhv^{-1}=\uhv^{0}=\uv_{0h}$, being $\uv_{0h}$ the initial condition, so that the time scheme reduces to the semi-implicit Euler method for the first time step $(\Delta t)^{0}=(2/3)\Delta t$. In the FOM, an impulsive start is performed, i.e. the initial condition is a zero velocity field, and the time step is $\Delta t = 2\cdot 10^{-3}\,\rm{s}$. 
Time integration is performed until a final time $T=7\,\rm{s}$. In the time period $[0,5]\,\rm s$, after an initial spin-up, the flow is expected to develop to full extent, including a subsequent relaxation time. Afterwards, it reaches a periodic-in-time (statistically- or quasi-steady) state.

The POD modes are generated in $L^2$ by the method of snapshots with velocity centered-trajectories \cite{IliescuJohn15} by storing every FOM solution from $t=5\,\rm{s}$, when the solution had reached a periodic-in-time state, and just using one period of snapshot data. The full period length of the statistically steady state is $0.332\,\rm{s}$, thus we collect $167$ velocity snapshots to generate the POD basis ${\cal \bU}^r$. 

\medskip 

{\em Numerical results for grad-div-ROM with different discretization of the nonlinear term.}
With POD velocity modes generated, the fully discrete grad-div-ROM \eqref{eq:pod_method2} is constructed as discussed in the previous section, using the semi-implicit BDF2 time scheme as for the FOM, and run with both skew-symmetric (as for the FOM) and EMAC (different from the FOM) formulation of the nonlinear term in the stable response time interval $[5,7]\,\rm{s}$ with $\Delta t=2\cdot 10^{-3}\,\rm{s}$. % and a small number ($r=8$) of POD velocity modes, which already gives a reasonable accuracy for the proposed method. 
The initial reduced-order velocity is given by the $L^2$-orthogonal projection $P_r^v$ of the velocity snapshot at $t=5\,\rm{s}$ on the POD velocity space ${\cal \bU}^r$. 
%
%A constant grad-div parameter $\mu$ has been chosen to run also the grad-div-ROM \eqref{eq:pod_method2} with both skew-symmetric and EMAC form of the nonlinear term, but fixed in this case minimizing the $L^{\infty}$ error in time with respect to the snapshots energy computed in one period and then repeated in the rest of periods, thus being the snapshots data to construct the reduced basis sufficient to compute the online constant $\mu$, and no further information is needed. We found that for the skew-symmetric grad-div-ROM is beneficial to consider a positive value of $\mu=3.7$ to improve the ROM accuracy (especially in the predictive time interval $[5.332,7]\,\rm s$), while for the EMAC grad-div-ROM a null value of $\mu$ guarantees the minimum error, and increasing it leads to slightly higher error levels, as we will see in the following numerical results. 

First of all, to assess on the one hand the behavior of the proposed grad-div-ROM \eqref{eq:pod_method2} with both skew-symmetric and EMAC form of the nonlinear term, and illustrate on the other hand the theoretical convergence order predicted by the numerical analysis performed in section \ref{sec:NA}, we plot the discrete $\ell^2(L^2)$ squared errors in velocity with respect to the grad-div-FEM \eqref{eq:gal_grad_div} solution. In particular, in Figure \ref{fig:plotl2L2Err1stPer} we show the errors in the reconstructive time interval $[5,5.332]\,\rm{s}$ (used to compute the snapshots and generate the POD modes) in terms of $r$ (number of POD modes). The theoretical analysis proved that, for sufficiently small $h$ and $\Delta t$ as it is the case, the velocity error should scale as $\Lambda_r=\sum_{k=r+1}^{d_v}\lambda_k$ (see result \eqref{eq:cota_finalSUPv}), and this is recovered in Figure \ref{fig:plotl2L2Err1stPer} for both discretizations of the nonlinear term. Following the hints given by the error bound \eqref{eq:cota_finalSUPv}, for the current setup for which $S^v_2=\| S^v\|_0 = 2.17 \cdot  10^2$, we actually found that $S^v_2\Lambda_r$ is an upper bound for the velocity error. Also, we note that when using a different online discretization of the nonlinear term (EMAC) with respect to the offline phase (skew), the ROM error is barely affected, as suggested by the theoretical analysis.

\begin{figure}[htb]
\begin{center}
\includegraphics[width=4.75in]{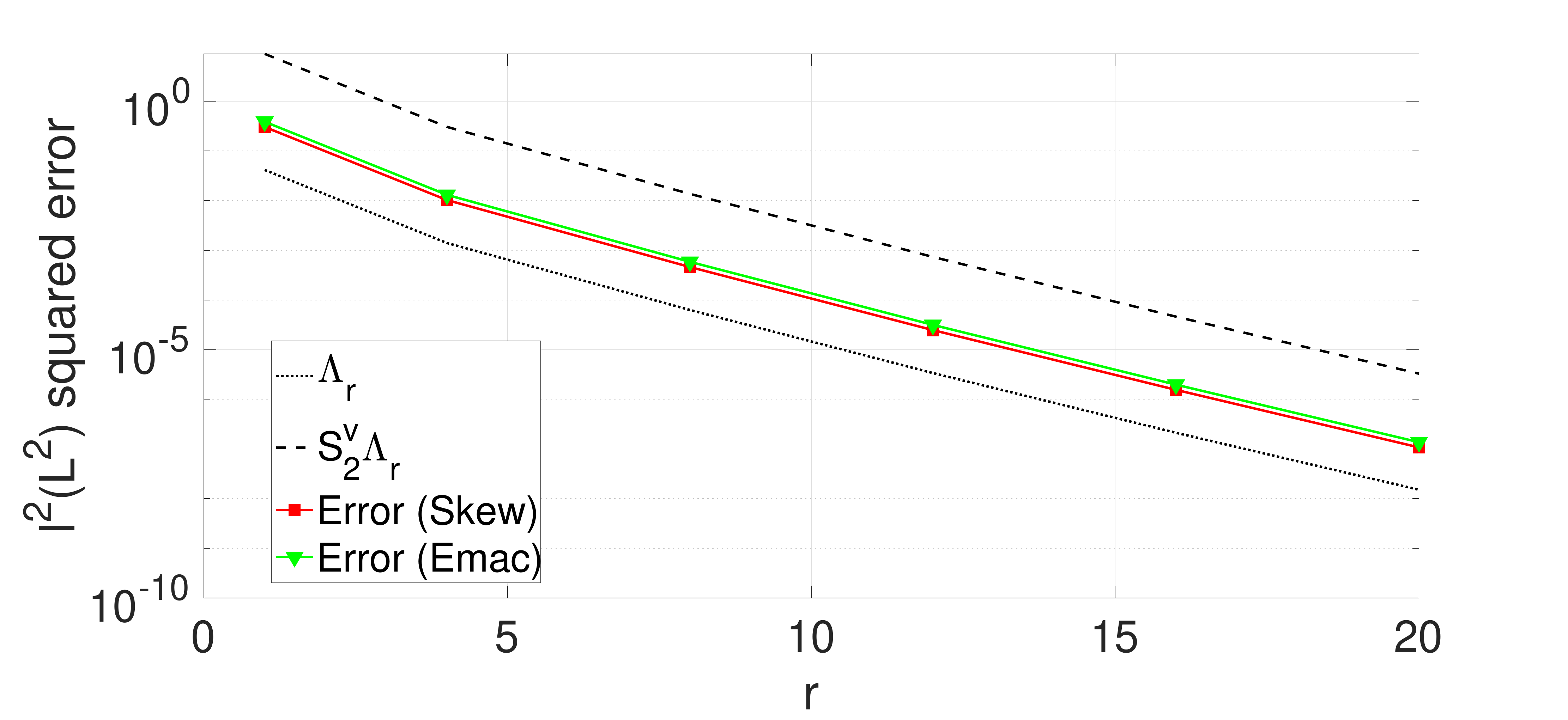}
\caption{Discrete $\ell^2(L^2)$ squared error in velocity with respect to grad-div-FEM \eqref{eq:gal_grad_div} computed with different online discretizations of the nonlinear term in the reconstructive time interval $[5,5.332]\,\rm{s}$.}\label{fig:plotl2L2Err1stPer}
\end{center}
\end{figure}

In Figure \ref{fig:plotl2L2Err}, we show the errors in the stable response time interval $[5,7]\,\rm{s}$. Thus, we are actually testing the ability of the skew-symmetric and EMAC grad-div-ROMs to predict/extrapolate in time, since we are monitoring their performance over an interval six times larger than the one used to compute the snapshots and generate the POD modes. In this case (not contemplated by the analysis), we notice that the errors increase and, for $r>12$, we observe a flattening effect due to the fact that the time interval $[5,7]\,\rm{s}$ is already quite
large with respect to the time period used to generate the POD basis. Thus, although we increase the number of POD modes, we do not notice so much the error decrease. In any case, for small $r$, the theoretical rate of convergence is recovered, and again very similar results are obtained for both online discretizations of the nonlinear term.

\begin{figure}[htb]
\begin{center}
\includegraphics[width=4.75in]{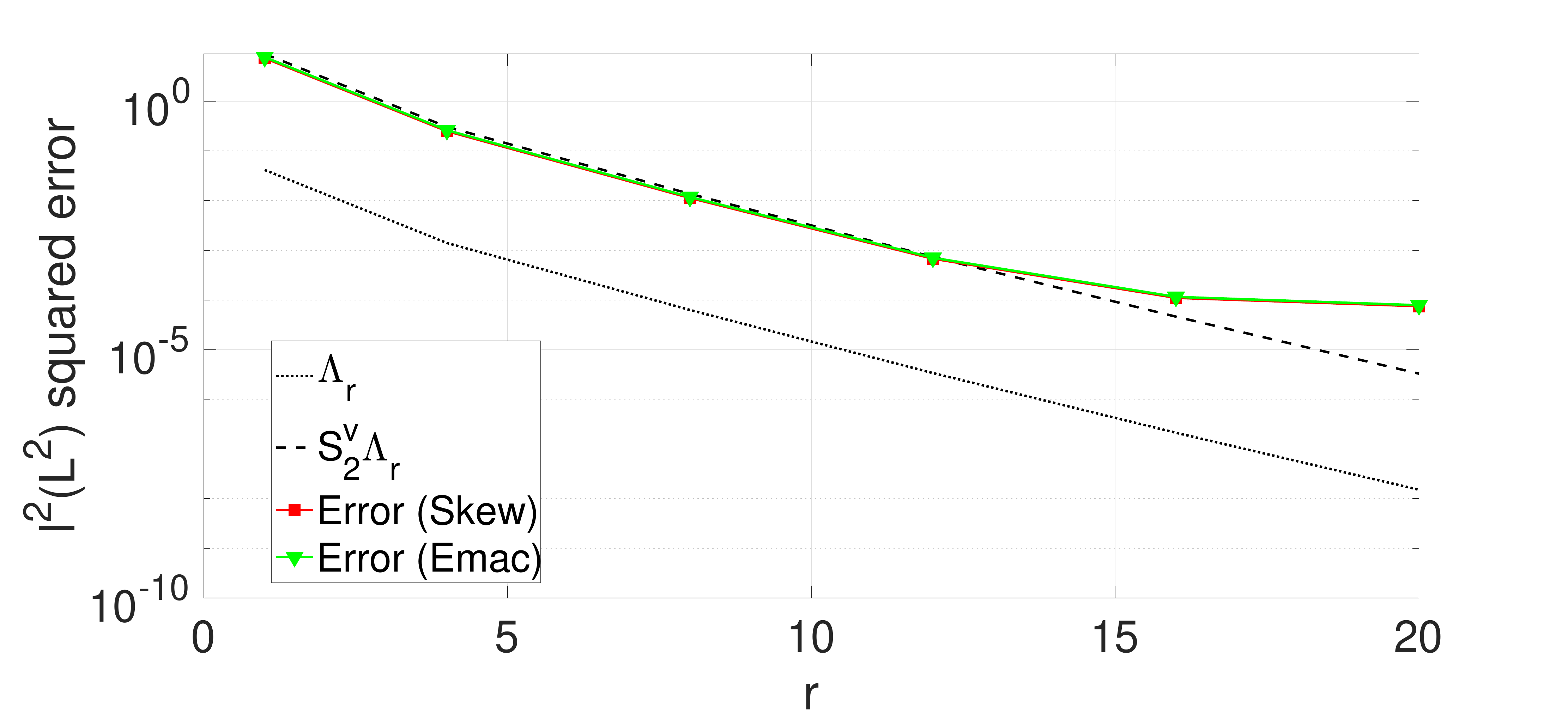}
\caption{Discrete $\ell^2(L^2)$ squared error in velocity with respect to grad-div-FEM \eqref{eq:gal_grad_div} computed with different online discretizations of the nonlinear term in the whole time interval $[5,7]\,\rm{s}$.}\label{fig:plotl2L2Err}
\end{center}
\end{figure}

To further assess the numerical accuracy of the grad-div-ROM \eqref{eq:pod_method2} with both skew-symmetric and EMAC form of the nonlinear term, the temporal evolution of the local drag and lift coefficients, and kinetic energy are monitored and compared to the FOM solutions in the stable response time interval $[5,7]\,\rm{s}$. To compute drag and lift coefficients, we used the volume integral formulation from \cite{John04paper}, where the pressure term is not
necessary if the test functions are taken properly in the discrete divergence-free space ${\boldsymbol V}_{h,l}$ \eqref{eq:V} (by Stokes projection, for instance), as done in \cite{pod_da_nos,zerfas_et_al}.

Numerical results for drag and lift predictions using $r=8$ velocity modes in the whole time interval $[5,7]\,\rm{s}$ are shown in Figures \ref{fig:DragCoefCompl}-\ref{fig:LiftCoefCompl}, where we display a comparison of grad-div-FEM \eqref{eq:gal_grad_div} with the skew-symmetric form of the nonlinear term, and grad-div-ROM \eqref{eq:pod_method2} with both skew-symmetric and EMAC form of the nonlinear term. %, and different values of the online grad-div coefficient $\mu$. 

From these figures, we see that both online discretizations of the nonlinear term agree reasonably well with the FOM data, although using a different formulation of the nonlinear term (EMAC) leads to a slightly inaccurate decrease of the drag coefficient as time increases. Similar conclusions can be obtained from Figure \ref{fig:LiftCoefCompl} for the lift coefficient. The good match that we observe using the same discretization for the nonlinear term as for the FOM (skew) is gradually lost using a different formulation of the nonlinear term (EMAC) in the online phase, which leads to a slightly inaccurate decrease of the lift coefficient too as time increases.

\begin{figure}[htb]
\begin{center}
\includegraphics[width=4.75in]{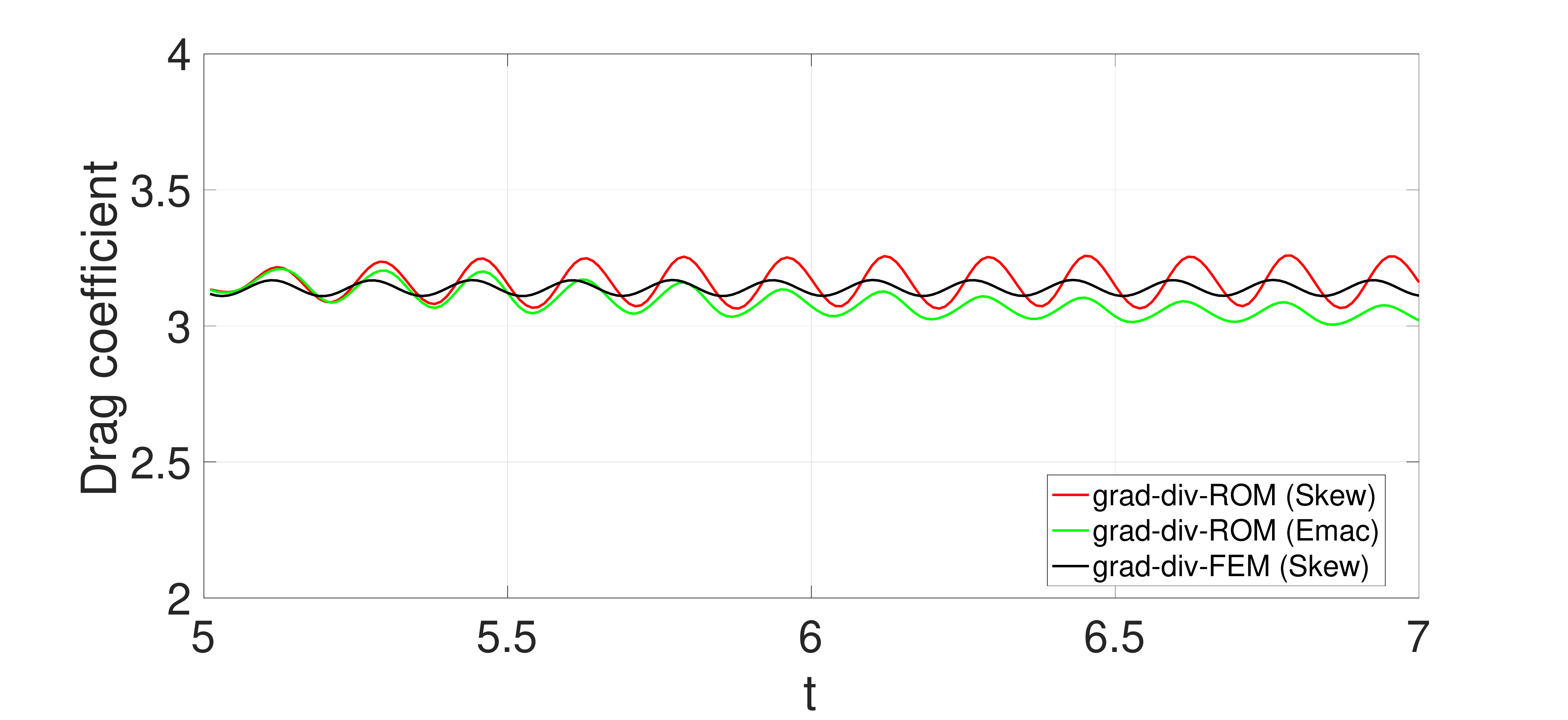}
\caption{Temporal evolution of drag coefficient computed with grad-div-ROM \eqref{eq:pod_method2} using $r=8$ velocity modes. Comparison with grad-div-FEM \eqref{eq:gal_grad_div} in the whole time interval $[5,7]\,\rm{s}$.}\label{fig:DragCoefCompl}
\end{center}
\end{figure}

\begin{figure}[htb]
\begin{center}
\includegraphics[width=4.75in]{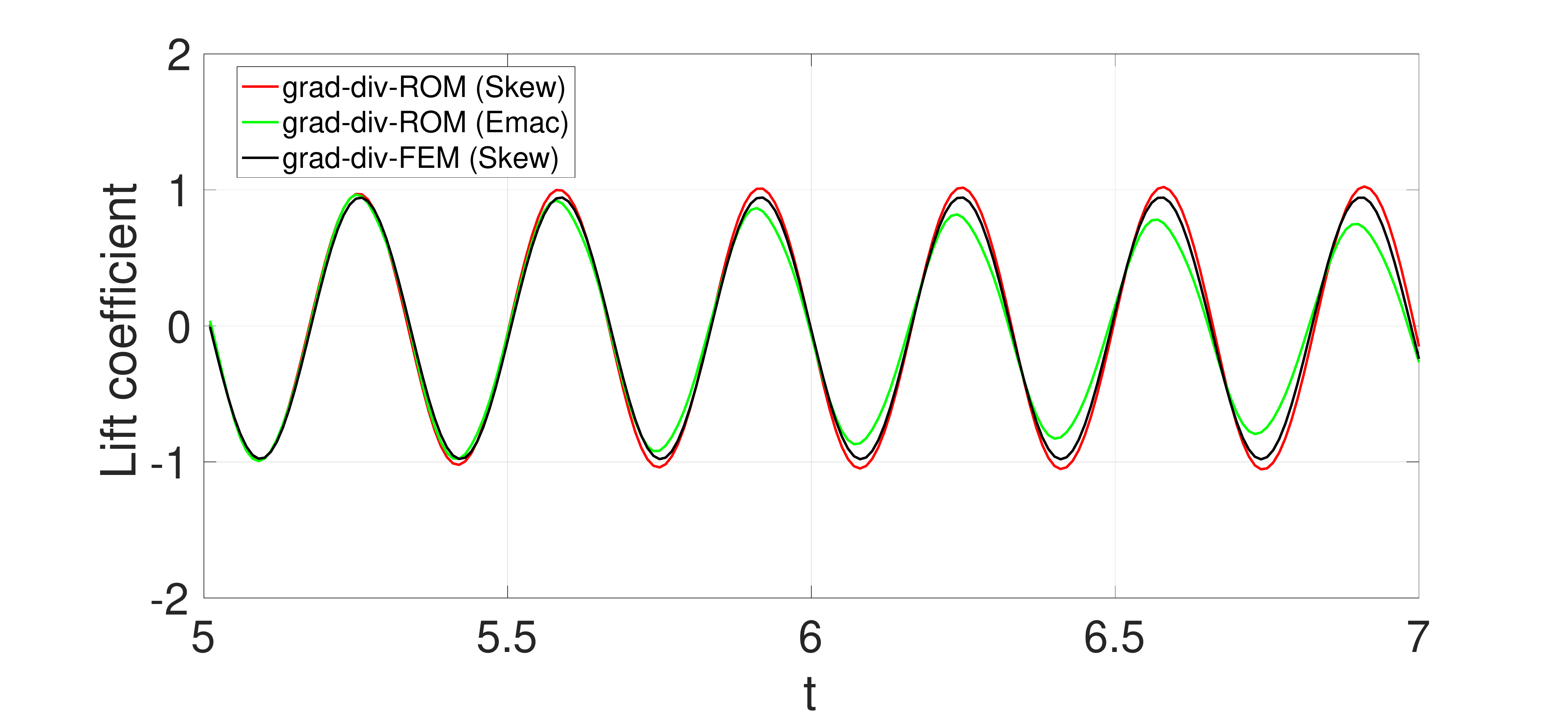}
\caption{Temporal evolution of lift coefficient computed with grad-div-ROM \eqref{eq:pod_method2} using $r=8$ velocity modes. Comparison with grad-div-FEM \eqref{eq:gal_grad_div} in the whole time interval $[5,7]\,\rm{s}$.}\label{fig:LiftCoefCompl}
\end{center}
\end{figure}

In Figure \ref{fig:EkinCompl}, we show the temporal evolution of the kinetic energy using $r=8$ velocity modes for the different grad-div-ROMs compared to grad-div-FEM in the whole time interval $[5,7]\,\rm{s}$. We observe that when considering in the ROM the skew-symmetric form of the nonlinear term (as for the FOM), the kinetic energy is maintained stably around the FOM values, while a slightly inaccurate decrease of the kinetic energy is observed with a different online formulation of the nonlinear term (EMAC) in the predictive time interval $[5.332,7]\,\rm{s}$. To sum up, using the same discretization of the nonlinear term in the ROM as in the FOM seems to perform slightly better for the considered quantities of interest, providing longer time stability and accuracy, while having very little effect on global error evaluations shown above, as predicted by the numerical analysis.

\begin{figure}[htb]
\begin{center}
\includegraphics[width=4.75in]{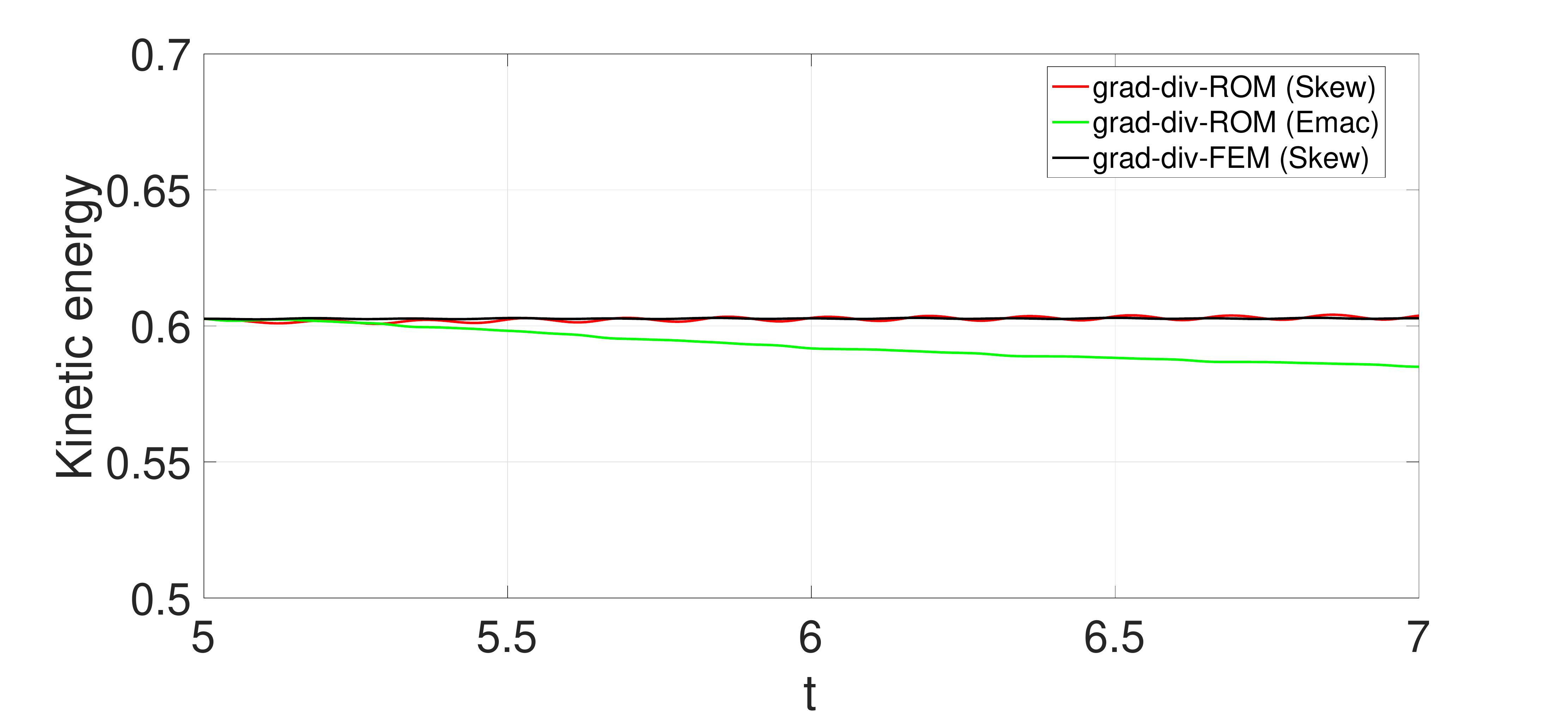}
\caption{Temporal evolution of kinetic energy computed with grad-div-ROM \eqref{eq:pod_method2} using $r=8$ velocity modes. Comparison with grad-div-FEM \eqref{eq:gal_grad_div} in the whole time interval $[5,7]\,\rm{s}$.}\label{fig:EkinCompl}
\end{center}
\end{figure}

%%%=====================================================
%%%=====================================================

\section{Conclusions}
As a conclusion, we can say that there is no inconsistency in the error bounds using different discretization for the nonlinear terms in the FOM and POD methods. In the case in which grad-did stabilization is added to both the FOM and POD methods error bounds with constants independent of inverse powers of the viscosity can be obtained. Comparing the case in which the same discretization of the nonlinear term is used for both FOM and POD methods we have checked that there is an added term in the error that has the size of the error of the FOM method
(in the $L^2$ norm of the velocity and the $L^2$ norm of the divergence of the velocity). 

On the other hand, in the case in which no stabilization is added neither to the FOM nor to the POD method, we have carried out two different error analysis. In the first one, the
error bounds depend on inverse powers of the viscosity and on the $L^2$ error of the velocity
and the $L^2$ error of the divergence of the velocity of the FOM method plus the $L^2$ error of the projection
of the square modulus of the velocity onto the FEM pressure space. In the second one, we are able to prove error bounds with constants
independent of inverse powers of the viscosity but with the price of having bounds that depend on the $H^1$ norm of the FEM velocity error
and the $H^1$ error of the projection
of the square modulus of the velocity onto the FEM pressure space. Moreover, since the bounds for the plain Galerkin method depend on inverse powers of the viscosity the dependence on $\nu^{-1}$ cannot completely be avoided also in this case. 

Overall, we have identified all the terms in the error bounds
for different combinations of discretizations of the nonlinear term adding or not grad-div stabilization.


\begin{thebibliography}{10}

\bibitem{Adams}
R.~A. Adams.
\newblock {\em Sobolev spaces}.
\newblock Academic Press [A subsidiary of Harcourt Brace Jovanovich,
  Publishers], New York-London, 1975.
\newblock Pure and Applied Mathematics, Vol. 65.

\bibitem{AhmedRubino19}
N.~Ahmed and S.~Rubino.
\newblock Numerical comparisons of finite element stabilized methods for a 2{D}
  vortex dynamics simulation at high {R}eynolds number.
\newblock {\em Comput. Methods Appl. Mech. Engrg.}, 349:191--212, 2019.

\bibitem{BF}
F.~Brezzi and R.~S. Falk.
\newblock Stability of higher-order {H}ood-{T}aylor methods.
\newblock {\em SIAM J. Numer. Anal.}, 28(3):581--590, 1991.

\bibitem{emac}
S.~Charnyi, T.~Heister, M.~A. Olshanskii, and L.~G. Rebholz.
\newblock On conservation laws of {N}avier-{S}tokes {G}alerkin discretizations.
\newblock {\em J. Comput. Phys.}, 337:289--308, 2017.

\bibitem{emac_fully}
S.~Charnyi, T.~Heister, M.~A. Olshanskii, and L.~G. Rebholz.
\newblock Efficient discretizations for the {EMAC} formulation of the
  incompressible {N}avier-{S}tokes equations.
\newblock {\em Appl. Numer. Math.}, 141:220--233, 2019.

\bibitem{Cia78}
P.~G. Ciarlet.
\newblock {\em The finite element method for elliptic problems}, volume~40 of
  {\em Classics in Applied Mathematics}.
\newblock Society for Industrial and Applied Mathematics (SIAM), Philadelphia,
  PA, 2002.
\newblock Reprint of the 1978 original [North-Holland, Amsterdam; MR0520174 (58
  \#25001)].

\bibitem{NS_grad_div}
J.~de~Frutos, B.~Garc\'{\i}a-Archilla, V.~John, and J.~Novo.
\newblock Analysis of the grad-div stabilization for the time-dependent
  {N}avier-{S}tokes equations with inf-sup stable finite elements.
\newblock {\em Adv. Comput. Math.}, 44(1):195--225, 2018.

\bibitem{review}
B.~Garc\'{\i}a-Archilla, V.~John, and J.~Novo.
\newblock On the convergence order of the finite element error in the kinetic
  energy for high {R}eynolds number incompressible flows.
\newblock {\em Comput. Methods Appl. Mech. Engrg.}, 385:Paper No. 114032, 54,
  2021.

\bibitem{pod_da_nos}
B.~Garc\'{\i}a-Archilla, J.~Novo, and S.~Rubino.
\newblock Error analysis of proper orthogonal decomposition data assimilation
  schemes with grad-div stabilization for the {N}avier-{S}tokes equations.
\newblock {\em J. Comput. Appl. Math.}, 411:114246, 2022.

\bibitem{IliescuJohn15}
S.~Giere, T.~Iliescu, V.~John, and D.~Wells.
\newblock S{UPG} reduced order models for convection-dominated
  convection-diffusion-reaction equations.
\newblock {\em Comput. Methods Appl. Mech. Engrg.}, 289:454--474, 2015.

\bibitem{Hecht12}
F.~Hecht.
\newblock New development in freefem++.
\newblock {\em J. Numer. Math.}, 20(3-4):251--265, 2012.

\bibitem{hey_ran_IV}
J.~G. Heywood and R.~Rannacher.
\newblock Finite-element approximation of the nonstationary {N}avier-{S}tokes
  problem. {IV}. {E}rror analysis for second-order time discretization.
\newblock {\em SIAM J. Numer. Anal.}, 27(2):353--384, 1990.

\bibitem{Ingimarson+2021}
S.~Ingimarson.
\newblock An energy, momentum, and angular momentum conserving scheme for a
  regularization model of incompressible flow.
\newblock {\em J. Numer. Math.}, 30(1):1--22, 2022.

\bibitem{archiv_incon}
S.~Ingimarson, L.~G. Rebholz, and T.~Iliescu.
\newblock Full and reduced order model consistency of the nonlinearity
  discretization in incompressible flows.
\newblock {\em arXiv:2111.06749v1 [math.NA] 12Nov2021}.

\bibitem{John04paper}
V.~John.
\newblock Reference values for drag and lift of a two-dimensional
  time-dependent flow around a cylinder.
\newblock {\em Internat. J. Numer. Methods Fluids}, 44:777--788, 2004.

\bibitem{John}
V.~John.
\newblock {\em Finite element methods for incompressible flow problems},
  volume~51 of {\em Springer Series in Computational Mathematics}.
\newblock Springer, Cham, 2016.

\bibitem{schneier}
K.~Kean and M.~Schneier.
\newblock Error analysis of supremizer pressure recovery for {POD} based
  reduced-order models of the time-dependent {N}avier-{S}tokes equations.
\newblock {\em SIAM J. Numer. Anal.}, 58(4):2235--2264, 2020.

\bibitem{kunisch}
K.~Kunisch and S.~Volkwein.
\newblock Galerkin proper orthogonal decomposition methods for parabolic
  problems.
\newblock {\em Numer. Math.}, 90(1):117--148, 2001.

\bibitem{novo_rubino}
J.~Novo and S.~Rubino.
\newblock Error analysis of proper orthogonal decomposition stabilized methods
  for incompressible flows.
\newblock {\em SIAM J. Numer. Anal.}, 59(1):334--369, 2021.

\bibitem{ols_reb}
M.~A. Olshanskii and L.~G. Rebholz.
\newblock Longer time accuracy for incompressible {N}avier-{S}tokes simulations
  with the {EMAC} formulation.
\newblock {\em Comput. Methods Appl. Mech. Engrg.}, 372:113369, 17, 2020.

\bibitem{SchaferTurek96}
M.~Sch\"{a}fer and S.~Turek.
\newblock Benchmark computations of laminar flow around a cylinder.
\newblock In E.~H. Hirschel, editor, {\em Flow {S}imulation with
  {H}igh-{P}erformance Computers II}, volume~48 of {\em Notes on {N}umerical
  {F}luid {M}echanics}, pages 547--566. Vieweg, 1996.

\bibitem{hood0}
C.~Taylor and P.~Hood.
\newblock A numerical solution of the {N}avier-{S}tokes equations using the
  finite element technique.
\newblock {\em Internat. J. Comput. \& Fluids}, 1(1):73--100, 1973.

\bibitem{zerfas_et_al}
C.~Zerfas, L.~G. Rebholz, M.~Schneier, and T.~Iliescu.
\newblock Continuous data assimilation reduced order models of fluid flow.
\newblock {\em Comput. Methods Appl. Mech. Engrg.}, 357:112596, 18, 2019.

\end{thebibliography}
\end{document}